\documentclass[12pt]{amsart}
\usepackage{latexsym,amssymb,amsfonts,amsmath, amscd, euscript, MnSymbol}
\usepackage{pdfsync}
\usepackage{enumerate}
\usepackage[pdftex]{graphicx}
\usepackage{wrapfig}
\usepackage{epsfig}
\usepackage{pgf,tikz}
\usetikzlibrary{arrows,automata}
\usetikzlibrary{positioning,calc}
\usepackage{amssymb, amsmath, amsthm, enumerate}
\usepackage{mathrsfs}
\usepackage{bbold}
\usepackage{url}
\usepackage{hyperref}
\usepackage{epstopdf}
\usepackage{enumerate}
\usepackage{verbatim}
\usepackage{float}
\usepackage{layout}
\usepackage[top=3cm, bottom=3cm, left=2.5cm, right=2.5cm]{geometry}
\usepackage{mathtools}

\DeclarePairedDelimiter\floor{\lfloor}{\rfloor}

\newtheorem{theorem}{Theorem}[section]

\newtheorem{corollary}[theorem]{Corollary}

\newtheorem{definition}[theorem]{Definition}
\newtheorem{example}[theorem]{Example}

\newtheorem{proposition}[theorem]{Proposition}

\newtheorem{remark}[theorem]{Remark}

\usepackage{siunitx}
\usepackage{graphics,mwe}

\newsavebox{\hold}
\newlength{\holdht}

\newcommand{\R}{\mathbb{R}}
\begin{document}

\title[Discrete and semimetric fixed point theorems]{On discrete and semimetric fixed point theorems for nonexpansive mappings}

\author[A. El adraoui, M. Kabil, A. Kamous and S. Lazaiz]{A. El adraoui$^1$, M. Kabil$^2$, A. Kamous$^3$ and S. Lazaiz$^3$}

\address{$^1$ Laboratory of Analysis Modeling and Simulation, Department of Mathematics and Computer Sciences\\ Faculty of Sciences Ben M'sick\\ University Hassan II, Casablanca, Morocco\\}
\address{$^2$Laboratory of Mathematics, Computer Science and Applications,\\ Faculty of Sciences and Technologies Mohammedia,\\ Hassan II University  of Casablanca, Morocco.\\}
\address{$^3$ENSAM Casablanca, Hassan II University, Casablanca, Morocco \\}

\email{samih.lazaiz@usmba.ac.ma}


\begin{abstract}
This work is a comparative study between the existence of fixed point for homomorphisms in a class of binary relationnal systems and   the existence of fixed point for nonexpansive mappings in semimetric spaces.
\end{abstract}

\maketitle
\section{Introduction}
The classical fixed point theorems of Banach and Tarski
marked the development of the two most prominent facets of the theory, namely, the metric fixed
point theory and the discrete fixed point theory.
\begin{definition}\rm\cite{granas2004}
A map $T:(X,d) \rightarrow (Y,\rho)$ of metric spaces that satisfies $\rho(T(x), T(z))\leq k\;d(x, z)$ for some nonnegative constant $k$ and all $x, z \in X$ is called Lipschitzian and 
the smallest such $k$ is called the Lipschitz constant. If $k < 1$,
the map $T$ is called contractive with contraction constant $k$.
\end{definition}
\begin{theorem} \rm{(Banach Contraction Principle (BCP))}\\
  Let $(X, d)$ be a complete
metric space and $T:(X,d) \rightarrow (X,d)$ be a contractive map. Then $T$ has a unique
fixed point $x_0$ and for every $x\in X$, the sequence $\{T^n(x)\}_{n\in \mathbb{N}}$ converges to $x_0$.
\end{theorem}
Nonexpansive mappings are those maps which have Lipschitz constant equal to one. These mappings can obviously be viewed as a natural extension of contraction mappings. However, the fixed point problem for nonexpansive mappings differ sharply from that of the contraction mappings in the sense that additional structure of the domain set is needed to insure the exsistence of fixed points.
The first existence results for nonexpansive mappings were obtained in 1965, independently,  by Browder \cite{browder1965}, G\"{o}hde \cite{gohde1965}. They proved that Every nonexpansive self-mapping of a closed convex and bounded
subset of a uniformly convex Banach space has a fixed point. Under assumptions slightly weaker, Kirk established the following result.
\begin{theorem}\cite{kirk1965}
  Let $C$ be a nonempty, bounded, closed and convex subset of a reflexive
Banach space X, and suppose that $C$ has normal structure. If $T$ is a mapping of $C$
into itself which does not increase distances, then $T$ has a fixed point. 
\end{theorem}
 In 1979, Penot \cite{ref10} generalizes Kirk's result in a metric space introducing a new convexity structure notion. In the setting of bounded hyperconvex metric spaces, Sine and Soardi \cite[1979]{sine1979,soardi1979}  proved that every non-expansive mapping has a fixed point.\vskip 2mm
Parallel to this development and after a series of research carried out by Knaster and Tarski between 1927 and 1955, the latter was able to establish this famous result of the fixed point for monotonic mappings on a poset:
\begin{theorem}\cite{tarski1955}
   Let $(E,\leq)$ be a complete lattice. If $T:E\rightarrow E$ is an order-preserving mapping, then $T$ admits a fixed point in $E$
\end{theorem}
  
The notable comparative study between the metric fixed point theory and the discrete one finds its roots in the works of Pouzet and his collaborators \cite{ref6,kabil16,kabil14,ref1}. In 1986, Jawhari et al. \cite{ref6} considered  a
generalized metric space , where, instead of real numbers, the distance values are members
of an ordered monoid equipped with an involution satisfying a distributivity condition. They showed that Tarski's classical fixed point theorem is closely related to Sine-Soardi theorem. Indeed, complete lattices may be seen as  hyperconvex metric spaces. 
\vskip 2mm
 In \cite{ref1}, Khamsi and Pouzet established an analogy between metric spaces and
binary relational systems. By extending  Penot's notions, they gave an abstract version of Kirk's result  in this framework. In \cite{Adra}, El adraoui et al. considered a generalized metric space $(E, d)$, where the metric takes their values in a binary relational system verifying some special conditions. Using a  generalization of the constructive lemma due to Gillespie and Williams
\cite{Gillepsie}, they proved that if the space $(E, d)$ has compact and normal
structure, then every nonexpansive mapping has a fixed point. They
obtained Tarski's fixed point theorem as a corollary. 
\vskip 2mm
In the same spirit,
Jachymski \cite{ref2} showed that Eilenberg theorem is equivalent to Banach Contraction Principle (BCP) restricted to ultrametric bounded spaces. Moreover, he gave an extension of Eilenberg's theorem and showed that this extension implies (BCP). At the end of his paper, Jachymski raised the problem of finding a discrete result that is equivalent to the (BCP).In \cite{ref3}, the authors showed that the given extension is equivalent to (BCP). 

\medskip
In this work, we propose a study of the fixed point results for nonxpansive mappings and for relational homomorphisms from this point of view, that is a comparative study between metric fixed point results and discrete fixed point results. For this purpose, we use the ideas developed by Pouzet in these various papers (see for instance \cite{ref6,ref7,ref1}).

\section{Preleminaries}
Let $M$ be a nonempty set. A binary relation on $M$ is any subset $R$ of $M\times M$. The inverse of $R$ is the binary relation $R^{-1}:= \{(x, y) : (y, x) \in R\}$. The diagonal is $\Delta := \{(x, x) : x \in M\}$. The relation $R$ is symmetric if $R = R^{-1}$. $R$ is reflexive if $\Delta\subset R$. 

\medskip
Throughout this paper, we consider a binary relational system $\mathcal{E}=\{R_n\}_{n\in \mathbb{Z}}$  over $M$ satsfying the following properties: 
\begin{enumerate}[($r_1$)]
  \item $R_n$ is symmetric for all $n\in \mathbb{Z}$;
  \item $R_n \subseteq R_{n-1}$ for all $n\in \mathbb{Z}$;
  \item $\bigcup_{n\in \mathbb{Z}} R_n = M\times M$;
  \item $\bigcap_{n\in\mathbb{Z}} R_n = \Delta$.
\end{enumerate}

\begin{example}\label{ex1}\rm
Let $M=\mathbb{R}$ and set for each $n\in \mathbb{Z}$, 
$$R_n=\left\{(x,y):|x-y|\leq 1/2^n\right\}.$$
Then $\{R_n\}_{n\in \mathbb{Z}}$ satisfies $(r_1)-(r_4)$. 
\end{example}
Let $x\in M$, the $\mathcal{E}$-ball of center $x$ and radius $R_n$, is the set $$\mathcal{B}(x,R_n) = \{y\in M\; : \; (x, y)\in R_n\}.$$

For a subset $A$ of $M$, set 
\[
cov_\mathcal{E}\left(A\right)=\bigcap\left\{ B\left(x,R_n\right)\;:\;x\in A,\; A\subseteq B\left(x,R_n\right)\right\}. 
\]

We will say that $A$ is an $\mathcal{E}$-admissible set if $A=cov_\mathcal{E}\left(A\right)$, that is an intersection of $\mathcal{E}$-balls. The family of all $\mathcal{E}$-admissible subsets of $M$ will be denoted by $\mathcal{A}_\mathcal{E}\left(M\right)$.

\begin{definition}\rm
 
\begin{enumerate}
  \item[$(r_5)$] A subset $K$ of $M$ is said to be bounded if there exists $N\in \mathbb{Z}$ such that $K\times K\subset R_N$.
  \item[$(r_6)$] $(M,\mathcal{E})$ has a compact structure if $\mathcal{A}_\mathcal{E}\left(M\right)$ has the finite intersection property.
  \item[$(r_7)$] $(M,\mathcal{E})$ has the normal structure if any $A\in \mathcal{A}_\mathcal{E}\left(M\right)$ not reduced to one point, we have $r_\mathcal{E}(A) \neq \delta_\mathcal{E}(A)$.
\end{enumerate}
\end{definition}

\begin{example}\label{ex2}\rm
Let $M$ and $R_n$ be as Example \ref{ex1}, then $(M,d)$ has compact and normal structure.

Let $M=\mathbb{R}$ and set for each $n\in \mathbb{Z}$, 
$$R_n=\left\{(x,y):|x-y|\leq 1/2^n\right\}$$
and $\mathcal{E}:=\{R_n\}_{n\in\mathbb{Z}}$. Then $(M,\mathcal{E})$ has a compact normal structure. Indeed,

\begin{enumerate}[(i)]
  \item $R_n$ is symmetric for all $n\in \mathbb{Z}$.
  
  \medskip
  \item \textbf{Compact structure}. Let $\mathbb{B}_\mathcal{E}=\{\mathcal{B}(x,R_n):\,x\in E,\;n\in \mathbb{Z}\}$ where $$\mathcal{B}(x,R_n)=\{y\in E: (x,y)\in R_n\}.$$ 
  Note that :
  $$ \mathcal{B}(x,R_n)=B(x,1/2^n)=\left[x-1/2^n,x+1/2^n\right].$$
  
  Let $\mathcal{F}=\{B_i=[x_i-1/2^{n_i},x_i+1/2^{n_i}]\}_{i\in I}$ be a family of members of $\mathbb{B}_\mathcal{E}$ for some index set $I$. We have to show that :
  $$\forall J\subset I\;\text{finite}, \quad \bigcap_{j\in J} B_j\neq \varnothing \Rightarrow \bigcap_{i\in I} B_i\neq \varnothing.$$

Suppose that for all $J\subset I$ finite, $\bigcap_{j\in J} B_j\neq \varnothing$. Thus, each $x_l+1/2^{n_l}$ must be greater than $x_i-1/2^{n_i}$ for all $i,l\in I$. Consequently, $\{x_i-1/2^{n_i}\}_{i\in I}$ is majorized and $\{x_i+1/2^{n_i}\}_{i\in I}$ is minorized, thus $a=\sup \;\{x_i-1/2^{n_i}\}_{i\in I}$ and $b=\inf \;\{x_i+1/2^{n_i}\}_{i\in I}$ exist and we have $a\leq b$. So 
\begin{equation}\label{eq1}
\varnothing \neq [a,b] \subset \bigcap_i B_i.  
\end{equation}

\medskip
 \item \textbf{Normal structure}. Let $A=\bigcap_{i\in I} B_i$ not reduced to one point. In fact the inclusion in equation (\ref{eq1}) is an equality. Indeed, let $x\in A$ then 
 $$x_i-\frac{1}{2^{n_i}}\leq x \leq x_i+\frac{1}{2^{n_i}}$$
 for each $i\in I$. Thus $a\leq x \leq b$. Hence $A=[a,b]$. Let $m$ be in $\mathbb{Z}$ such that 
 \begin{equation}
   m\leq \frac{\log\left(\frac{2}{b-a}\right)}{\log(2)}=1+\frac{\log\left(\frac{1}{b-a}\right)}{\log(2)}.
 \end{equation}
 
 Then, we get 
 $$A\subseteq \left[\frac{a+b}{2}-\frac{1}{2^{m}},\frac{a+b}{2}+\frac{1}{2^{m}}\right]=B\left(\frac{a+b}{2},\frac{1}{2^{m}}\right)=\mathcal{B}\left(\frac{a+b}{2},R_m\right).$$
 
 Set $\alpha=\frac{\log\left(\frac{1}{b-a}\right)}{\log(2)}$ and choose $m=1+\floor{\alpha}$. Thus we have $\alpha<m$ which is equivalent to 
 $$
 \begin{array}{ccl}
   \frac{\log\left(\frac{1}{b-a}\right)}{\log(2)} < m &\Leftrightarrow& \log\left(\frac{1}{b-a}\right)<\log(2^m)\\
                                                      & \Leftrightarrow & \frac{1}{b-a}<2^m\\
                                                      & \Leftrightarrow & b-a > \frac{1}{2^m}
                                                   
 \end{array}
 $$

With this in hand, we have $R_m\in r(A)$, and since for $x=a\in A$ and $y=b\in A$ we have 
 $$|x-y|>\frac{1}{2^{m}}$$ 
 then $A\times A \not \subset  R_m$ i.e. $R_m \notin \delta(A)$. Thus $\delta(A)\subsetneq r(A)$
\end{enumerate}
\end{example}

\begin{definition}
Let $M$ be a nonempty set and $(R_n)_n$ a binary relational system. Let $T:M\rightarrow M$ be a self mapping, we say that $T$ is relational homomorphisms if:
\begin{itemize}
    \item[$(r_8)$] given $n \in\mathbb{Z}$, if $(x, y) \in R_n$ then $(Tx, Ty) \in R_n$. 
\end{itemize}
\end{definition}

The following theorem is a particular case for countable family of relational system of Khamsi-Pouzet abstract version.

\begin{theorem}\label{Thm2}\cite[Theorem 2.9]{ref1}
  Let $M$ be a nonempty bounded set, $(R_n)_{n\in\mathbb{Z}}$ a sequence of reflexive binary relations over $M$ and $T$ a self-map of $M$ such that $(r_1)-(r_8)$ hold. Then $T$ has a fixed point.
\end{theorem}

\begin{example}
Let $M$ and $R_n$ be as the Example \ref{ex2}. Consider $T: M\rightarrow M$ to be the mapping defined by $Tx= x+1$. It is a simple verification that $T$ satisfies condition $(r_8)$. All assumptions of Theorem \ref{Thm2} are satisfied except that $M$ is a bounded set, however $T$ does not have a fixed point. 
\end{example}
\begin{definition}\label{Def1}
Let $M$ be a nonempty set. A mapping $d: M\times M \rightarrow \R$ is called a semimetric over $M$ if it satisfies
the following conditions:

\begin{enumerate}[(d1)]
  \item $d(x, y) \geq 0$ for all $x, y \in  M$, and $d(x, y) = 0$ if and only if $x = y$;
  \item $d(x, y) = d(y, x)$ for all $x, y \in M$.
\end{enumerate}
The pair $(X, d)$ is called semimetric space.
\end{definition}

Let $x\in M$, the ball of center $x$ and radius $r>0$, is the set 
$$B(x,r) = \{y\in M\; : \; d(x, y)\leq r\}.$$

For a subset $A$ of a semimetric space $\left(M,d\right)$, set  
\[
cov\left(A\right)=\bigcap\left\{ B\left(x,r\right)\;:\;x\in A,\; A\subseteq B\left(x,r\right)\right\} 
\]
we will say that $A$ is an admissible set if $A=cov\left(A\right)$, that is an intersection of balls. The family of all admissible subsets
of $M$ will be denoted by $\mathcal{A}\left(M\right)$. For a bounded subset $A \subset M$, we set

\begin{itemize}
  \item $r_x(A):=\sup\{d(x,y):\; y\in A\}$;
  \item $r_d(A):=\inf\{r_x(A):\; x\in A\}$;
  \item $\delta_d(A):=\sup\{d(x,y):\; x,y\in A\}$;
\end{itemize}

\begin{definition}\label{def1}\rm
Let $(M, d)$ be a semimetric space.
\begin{enumerate}[(i)]
  \item  $(M, d)$ has a compact structure if $\mathcal{A}\left(M\right)$ has the finite intersection property, i.e., for every family $\mathcal{F}$ of members of $\mathcal{A}\left(M\right)$, the intersection of $\mathcal{F}$ is nonempty provided that the intersection of all finite subfamilies of $\mathcal{F}$ are nonempty.
  
  \item $(M,d)$ is said to have normal structure if any $A\in \mathcal{A}\left(M\right)$ not reduced to one point, we have $r_d(A) < \delta_d(A)$.
\end{enumerate}
\end{definition} 

The following theorem is a particular case for semimetric of this abstract version.

\begin{theorem}\label{Thm4}\cite[Theorem 3]{ref4}
Let $(M, d)$ be a nonempty bounded semimetric space such that $\mathcal{A}(M)$ has compact and normal structure. Then every nonexpansive mapping $T:M\rightarrow M$ has a fixed point.
\end{theorem}

\section{Equivalence between binary relational system and semimetric space}
The following result constitutes a bridge between binary relation system and semimetric spaces.
\begin{proposition}\label{Prop1}\rm
Let $M$ be a nonempty set and $(R_n)_{n\in\mathbb{Z}}$ a sequence of reflexive binary relations over $M$ satisfying the conditions $(r_1)-(r_4)$.\\ 
Define the mappings $\alpha : M\times M\rightarrow 2^\mathbb{Z}$, $\mu : M\times M\rightarrow \mathbb{Z}\cup \{\infty\}$, and $\delta: M\times M\rightarrow [0, \infty)$ as:
$$\alpha (x, y) = \{n \in \mathbb{Z} : (x, y) \in R_n\},\quad \mu(x, y) = \sup \alpha (x, y), \quad \delta(x, y) = 2^{-\mu(x,y)}$$

for each $(x, y) \in M\times M$ (here, by convention, $2^{-\infty}= 0$).

Then, $\delta$ is a semimetric over $M$  and for every $n \in \mathbb{Z}$,
\begin{equation}\label{eq1}
  R_n = \{(x, y) \in M\times M : \delta(x, y) \leq 2^{-n}\}
\end{equation}

\end{proposition}

\begin{proof}
The proof follows the same line as \cite[Proposition 3.1]{ref3}.

\end{proof}

\begin{theorem}\label{Thm1}\rm
  Let $T$ be a self-map of an abstract set $M$. The following statements are equivalent. 
  \begin{enumerate}[(i)]
    \item There exists a sequence $(R_n)_{n\in\mathbb{Z}}$ of binary relations in $M$ such that the assumptions $(r_1)-(r_4)$ are satisfied.
    \item There exists a semimetric $d$ such that $T$ is an nonexpansive with respect to $d$.
  \end{enumerate}
\end{theorem}

\begin{proof}
\begin{itemize}
  \item[(i)$\Rightarrow$(ii)] Using Propostion \ref{Prop1} and taking $d=\delta$ we get the first part of the statement. To see that $T$ is $d$-nonexpansive mapping, let $x,y\in M$, so $(x,y)\in R_{\mu(x,y)}$, then  by $(r_4)$ $(Tx,Ty)\in R_{\mu(x,y)}$. Hence, $\mu(Tx,Ty)\geq \mu(x,y)$, which is exactely $d(Tx,Ty)\leq d(x,y)$.
  \item[(i)$\Leftarrow$(ii)] Let $R_n = \{(x, y) \in M\times M : d(x, y) \leq 2^{-n}\}$ for every $n \in \mathbb{Z}$ then $(r_1)-(r_4)$ hold.
\end{itemize}
\end{proof}

\begin{theorem}
  Theorem \ref{Thm4} and Theorem \ref{Thm2} are equivalent.
\end{theorem}

\begin{proof}$\\$

\begin{itemize}
  \item[\ref{Thm4}$\Rightarrow$\ref{Thm2}] Under the assumptions of Theorem \ref{Thm2} and notations of Proposition \ref{Prop1}, define
  $$d(x, y) = 2^{-\mu(x,y)}$$ 
  for every $x,y\in M$. Let $n\in \mathbb{Z}$ and $x\in M$, we have 
  $$
  \begin{array}{ccl}
   y \in \mathcal{B}(x,R_n) & \Leftrightarrow & (x,y)\in R_n\\
                            & \Leftrightarrow & d(x,y) \leq \frac{1}{2^n}\\
                            & \Leftrightarrow & y\in B(x, \frac{1}{2^n})    
  \end{array}
  $$
  then $\mathcal{B}(x,R_n)=B(x,\frac{1}{2^n})$ and since $d(x,y) \in \{\frac{1}{2^n}: n\in \mathbb{Z}\}\cup\{0\}$ we have $\mathcal{A}\left(M\right)=\mathcal{A}_\mathcal{E}\left(M\right)$. 
  
  \medskip
  In addition, since $\mathcal{A}_\mathcal{E}\left(M\right)$ has compact structure then $\mathcal{A}\left(M\right)$ has compact structure too. 
  
  \medskip
  Now we show that $\mathcal{A}\left(M\right)$ has normal structure. Indeed, Let $A$ be a bounded subset of $\mathcal{A}\left(M\right)$, then
$$
\begin{array}{ccl}
  r_d(A)<\delta_d(A) & \Leftrightarrow & (\exists k,p\in \mathbb{Z})\; r_d(A)=\frac{1}{2^k}\;\text{and}\;\delta_d(A)=\frac{1}{2^p}\;\text{and}\, p<k\\
                     & \Leftrightarrow &  R_k \in r_\mathcal{E}(A)\;\text{and}\;R_k\notin \delta_\mathcal{E}(A)\\
                     & \Leftrightarrow & r_\mathcal{E}(A)\neq \delta_\mathcal{E}(A)
\end{array}
$$

\item[\ref{Thm2}$\Rightarrow$\ref{Thm4}] If we set as above $R_n = \{(x, y) \in M\times M : d(x, y) \leq 2^{-n}\}$ then all asumptions of Theorem \ref{Thm2} are satisfied and then we get the result.
\end{itemize}
\end{proof}
\section{Discussion}
\subsection{Inframetric spaces}

In latency-sensitive networks, the location of servers in relation to an Internet host is a question that has aroused the interest of a large number of researchers. Many algorithms have been designed for the Internet assuming that the distance defined by the round-trip delay (RTT) is a metric. However, performance analysis of several of these algorithms shows that the triangle inequality, which is not necessarily satisfied by the RTT distance. To better understanding of the Internet structure
at the router level, Fraigniaud et al. proposed in  \cite{infra} an analytic approach based on the inframetric distance.

\begin{definition}\rm
Let $M$ be a non empty set. Let $C\geq 1$. Consider the two conditions:
\begin{itemize}
  \item[($d_3$)] $d (x, y) \leq d (x, z) + d (z, y)$,  for each $x, y, z \in M$.
  \item[($d_4$)] $d (x, y) \leq C\cdot\max\{d (x, z),d (z, y)\}$, for each $x, y, z \in M$ .
\end{itemize}

\begin{itemize}
  \item A pair $(M, d)$ satisfying axioms ($d_1$), ($d_2$) and ($d_3$) is called a metric space.
  \item A pair $(M, d)$ satisfying axioms ($d_1$), ($d_2$) and ($d_4$) is called a $C$-inframetric space.
\end{itemize}
\end{definition}

\begin{example}\label{ex3}\rm
Let $M=\mathbb{R}$ and $(R_n)_{n\in\mathbb{Z}}$ as Example \ref{ex1} and let $\mu(x,y)=\sup\{n\in \mathbb{Z}:\;(x,y)\in R_n\}$ and define $d: M\times M \rightarrow \R$ by 
$$d(x,y)=2^{-\mu(x,y)},$$
then, $(M,d)$ is a $2$-inframetric space. 
\end{example}

\begin{remark}
Since the family $(R_n)$ is reflexive, we have for all $n\in \mathbb{Z}$ 
$$R_n\subseteq R^2_n\subseteq R^3_n\subseteq\cdots$$
\end{remark}

\begin{proposition}\cite[Proposition 3.1]{ref3}\label{Prop2}\rm
Let $M$, $(R_n)_{n\in\mathbb{Z}}$ and $\delta$ as Proposition \ref{Prop1} satisfying conditions $(r_1),(r_3),(r_4)$ and
\begin{itemize}
  \item[$(r_9)$] $R^2_n \subseteq R_{n-1}$ for all $n\in \mathbb{Z}$.
\end{itemize}
Then, $\delta$ is a $2$-inframetric over $M$ and equation (\ref{eq1}) holds.
\end{proposition}

\begin{theorem}\label{Thm8}
Let $(M, d)$ be a nonempty bounded $2$-inframetric space such that $\mathcal{A}(M)$ has compact and normal structure. Then every nonexpansive mapping $T:M\rightarrow M$ has a fixed point.
\end{theorem}

\begin{theorem}\label{Thm7}
  Let $M$ be a nonempty bounded set, $(R_n)_{n\in\mathbb{Z}}$ a sequence of reflexive binary relations over $M$ and $T$ a self-map of $M$ such that $(r_1)$ and $(r_3)-(r_9)$ hold. Then $T$ has a fixed point.
\end{theorem}

\begin{corollary}
Theorem \ref{Thm8} and Theorem \ref{Thm7} are equivalent.
\end{corollary}
\subsection{Metric spaces}
\begin{proposition}\cite[lemma 12, pp 185]{ref9}\label{Prop3}\rm
Let $M$, $(R_n)_{n\in\mathbb{Z}}$ and $\delta$ be as in  Proposition \ref{Prop1} satisfying conditions $(r_1),(r_3),(r_4)$ and
\begin{itemize}
  \item[$(r_{10})$] $R^3_n \subseteq R_{n-1}$ for all $n\in \mathbb{Z}$.
\end{itemize}
Then, $\delta$ is a metric over $M$ and equation (\ref{eq1}) holds.
\end{proposition}

\begin{theorem}\label{Thm5}\cite[Theorem 3.2]{ref10}
Let $(M, d)$ be a nonempty bounded metric space such that $\mathcal{A}(M)$ has compact and normal structure. Then every nonexpansive mapping $T:M\rightarrow M$ has a fixed point.
\end{theorem}

\begin{theorem}\label{Thm6}
  Let $M$ be a nonempty bounded set, $(R_n)_{n\in\mathbb{Z}}$ a sequence of reflexive binary relations over $M$ and $T$ a self-map of $M$ such that $(r_1)$, $(r_3)-(r_8)$ and $(r_{10})$ hold. Then $T$ has a fixed point.
\end{theorem}

\begin{corollary}
Theorem \ref{Thm5} and Theorem \ref{Thm6}  are equivalent.
\end{corollary}

\subsection{Ultrametric spaces}

\begin{definition}\cite{ref11}
Assume that $T:M\rightarrow M$ is a map and $B=\mathcal{B}(x,R_n)\subset M$ an $\mathcal{E}$-ball of $M$. We say that $B$ is \textit{minimal $T$-invariant} if : 
\begin{itemize}
  \item[(i)] $T(B)\subseteq B$, and
  \item[(ii)] $(y,Ty)\in R_n\setminus R_{n+1}$ for each $y\in B$. 
\end{itemize}
\end{definition}

\begin{remark}
Note that this definition is exactly the classical definition for standard ultrametric space. Indeed, since  $(y,Ty)\in R_n\setminus R_{n+1}$ we have 
$$\frac{1}{2^{n+1}}<d(y,Ty)\leq \frac{1}{2^n}$$ 
and since $d(x,y) \in \{\frac{1}{2^n}: n\in \mathbb{Z}\}\cup\{0\}$ we have $d(y,Ty)=\frac{1}{2^n}$.   
\end{remark}

\begin{definition}\cite{ref14,ref13}
An ultrametric space $(M,d)$ is said to be spherically complete if and only if for any sequence $\{B(x_n,r_n)\}_{n\in \mathbb{N}}$ of closed balls such that $r_{n+1}\leq r_n$ and $B(x_{n+1},r_{n+1})\subset B(x_n,r_n)$, for any $n\in\mathbb{N}$ we have $\bigcap_{n\in \mathbb{N}} B(x_n,r_n)\neq \varnothing$.
\end{definition}

\begin{theorem}\label{KS}\cite{ref11,ref12}
Suppose $M$ is a spherically complete ultrametric space and $T:M\rightarrow M$ is a nonexpansive map. Then every ball of the form $B(x,d (x,Tx))$ contains either a fixed point of $T$ or a minimal $T$-invariant ball. 
\end{theorem}

\begin{remark}
\begin{itemize}
  \item In \cite{ref15}, the authors showed that the spherical completness of an ultrametric space is exactely the compactness in the sense of Penot (see \cite[Theorem 5]{ref15}), that is the compactness of the collection of admissible sets.
  \item Since the diameter of a closed ball in an ultrametric space can never exceed its radius, $\mathcal{A}(M)$ is never normal for an ultrametric space $M$.
\end{itemize} 
\end{remark}

Take into account the above discussion, we obtain, the equivalence result of Theorem \ref{KS} for a family of equivalence relations

\begin{theorem}
  Let $M$ be a nonempty set, $(R_n)_{n\in\mathbb{Z}}$ a sequence of equivalence relations over $M$ and $T$ a self-map of $M$ such that $(r_3)-(r_5)$ and $(r_8),(r_{10})$ hold. Then every $\mathcal{E}$-ball contains either a fixed point of $T$ or a minimal $T$-invariant ball. 
\end{theorem}

\begin{proof}
With this hypothesis, $d$ is an ultrametric (see the proof of \cite[Theorem 3.1]{ref2}).
\end{proof}

\begin{definition}\label{Def12}
Let $M$ be a nonempty set and $(R_n)_{n\in\mathbb{Z}}$ a sequence of reflexive binary relations over $M$ such that conditions $(r_1)-(r_4)$ are satisfied. Let $T$ be a map from $M$ to itself. We will say that $T$ :
\begin{enumerate}
  \item[(a)] has the regular property if for all $x\in M$ with $x\neq Tx$ there exists $n_0\in\mathbb{N}$ such that 
\begin{equation}
  (T^nx,T^{n+1}x)\in R_{\mu(x,Tx)+n_0},\quad \forall n\geq n_0.
\end{equation}
 \item[(b)] is asymptotically regular if for all $x\in M$ with $x\neq Tx$ there exists $n_0\in\mathbb{N}$ such that 
\begin{equation}
  (T^nx,T^{n+1}x)\in R_{\mu(x,Tx)+n},\quad \forall n\geq n_0.
\end{equation}
\end{enumerate}
\end{definition}
Note that if $T$ is asymptotically regular at $x$ then it has regular property at $x$.
\begin{remark}
 Recall that in \cite{ref15}, the authors introduce the following notion. Let $(M,d)$ be a metric space. A self mapping $T$ is said to have the weak-regular property if 
$$\limsup_{n\rightarrow \infty}\ d(T^n(x),T^{n+1}(x))< d(x,T(x)),$$
for each $x$ in $M$ such that $x \neq T(x)$.

\medskip
Now, if $T$ has the regular property at $x$ then it has the weak-regular property at $x$ for the corresponding metric $d$. Indeed, there exists $n_0$ such that 
$$(T^nx,T^{n+1}x)\in R_{\mu(x,Tx)+n_0},$$
for all $n\geq n_0$. With the above, we have 
$$(T^nx,T^{n+1}x)\in R_{\mu(x,Tx)+n_0}\Leftrightarrow d(T^nx,T^{n+1}x)\leq \frac{1}{2^{\mu(x,Tx)+n_0}}$$
and since $\frac{1}{2^{\mu(x,Tx)}}=d(x,Tx)$ we get 
$$d(T^nx,T^{n+1}x)\leq \frac{1}{2} d(x,Tx).$$
Thus, $\limsup_{n\rightarrow \infty}\ d(T^n(x),T^{n+1}(x))< d(x,T(x)).$
\end{remark}

\begin{remark}
The asymptically regular of $T$ in the sens of definition \ref{Def12} is equivalent to the classical asymptically regular property for the corresponding metric $d$. Recall that $T$ is asymptotically regular at $x$ if and only if

$$\lim_{n\to\infty} d(T^nx, T^{n+1} x)= 0.$$
\end{remark}

Now we can give fixed point results for ultrametric spaces which are extention of Theorem 17 and Theorem 15 in \cite{ref15} respectively.

\begin{theorem}
  Let $M$ be a nonempty set, $(R_n)_{n\in\mathbb{Z}}$ a sequence of equivalence relations over $M$ and $T$ a self-map of $M$ with the regular property such that $(r_3)-(r_5)$ and $(r_8),(r_{10})$ hold. Then $T$ has a fixed point in any $T$-invariant $\mathcal{E}$-ball. 
\end{theorem}

\begin{theorem}
  Let $M$ be a nonempty set, $(R_n)_{n\in\mathbb{Z}}$ a sequence of equivalence relations over $M$ and $T$ a self-map of $M$ which is is asymptotically regular such that $(r_3)-(r_5)$ and $(r_8),(r_{10})$ hold. Then $T$ has a fixed point in any $T$-invariant $\mathcal{E}$-ball. 
\end{theorem}

\begin{example}
Let $M=\mathbb{R}\cup\{\infty\}$ and define for all $n\in\mathbb{Z}$ 
$$R_n=\{(x,y)\in M^2: \min\{x,y\}\geq n\}\cup\Delta.$$
Note that $\{R_n\}_{n\in\mathbb{Z}}$ is a family of an equivalence relations satisfying the conditions $(r_1)-(r_4)$. In addition, for all $x\in M$ and $n\in \mathbb{Z}$ we have
$$y\in B(x,R_n)\Leftrightarrow x=y\; \text{or}\; \min\{x,y\}\geq n.$$

\begin{itemize}
\item If $n>x$ then $B(x,R_n)=\{x\}.$
\item If $n\leq x$ then $B(x,R_n)=[n,\infty]$.
\end{itemize}

$\{R_n\}_{n\in\mathbb{Z}}$ has compact structure. Let $\mathcal{F}=\{B_i=B(x_i,R_{n_i})\}_{i\in I}$ be a family of balls of $M$. 
\begin{itemize}
\item[Case-1] If there exists $i_0\in I$ such that $n_{i_0}<x_{i_0}$, then $B_{i_0}=\{x_{i_0}\}$. 
\item[Case-2] Assume that $\forall i\in I$ we have $n_i\leq x_i$. Then, $B_i=[n_i,\infty]$, thus 
$$\infty \in \bigcap_{i\in I} B_i\neq \varnothing.$$

On the other hand, $\{R_n\}_{n\in\mathbb{Z}}$ does not have normal structure. Indeed, let $A$ be an intersection of balls not reduced to one point. Then, there exists $n\in A\cap\mathbb{Z}$ such that $A=[n,\infty]$. Thus, $R_m\in r_\mathcal{E}(A)$ if and only if $m\leq n$. Moreover, $A\times A\subset R_k$ if and only if $\forall x,y\in A$ we have $\min\{x,y\}\geq k$, then $k\leq n$. Hence, $r_\mathcal{E}(A)=\delta_\mathcal{E}(A)$.

In fact, $\{R_n\}_{n\in\mathbb{Z}}$ defines an ultrametric distance which implies derictely that $\{R_n\}_{n\in\mathbb{Z}}$ does not have normal structure. 
\end{itemize} 

Now define the mapping $T:M\rightarrow M$ to be 
$$
\begin{cases}
  Tx=x+1&\text{if}\; x\in \mathbb{R}\\
  T\infty=\infty&
\end{cases}
$$
 Then $T$ satisfies obviously $(r_8)$. In addition, $T$ is asymptotically regular, since $(x+n,x+n+1)\in R_ {\mu(x,x+1)+n}$ for all $x\in \mathbb{R}$. Hence $T$ has a fixed point in every $T$-invariant $\mathcal{E}$-ball. Note that every $T$-invariant $\mathcal{E}$-ball contains $\infty$.
\end{example}

\bibliographystyle{plain}

\end{document}